\documentclass{amsart}
\usepackage{amssymb,amsmath,amsthm,verbatim,color,soul}
\usepackage[all]{xy}

\usepackage{tikz}
\usepackage{tkz-euclide}

\def\bP{\mathbf{P}}
\def\bI{\mathbf{I}}
\def\bC{\mathbf{C}}
\def\bV{\mathbf{V}}
\def\cO{\mathcal{O}}
\def\cK{\mathcal{K}}
\def\cX{\mathcal{X}}
\def\c{\mathcal}

\def\Hom{\mathrm{Hom}}

\def\Syz{\mathrm{Syz}}
\def\Sym{\mathrm{Sym}}
\def\Im{\mathrm{Im}}
\def\Span{\mathrm{Span}}

\newtheorem{thm}{Theorem}[section]
\newtheorem{prop}[thm]{Proposition}
\newtheorem{lemma}[thm]{Lemma}

\newtheorem{definition}[thm]{Definition}

\newtheorem{rem}[thm]{Remark}
\newtheorem{qn}[thm]{Question}

\newcommand{\be}{\begin{equation}}
	\newcommand{\ee}{\end{equation}}

\title{The second syzygy schemes of curves of large degree}
\date{\today}
\subjclass[2000]{14H51; 13D02}
\keywords{Projective curves, Brill-Noether theory, syzygies, syzygy schemes}
\author{Marian Aprodu}
\address{University of Bucharest, Department of Mathematics \& Institute
	of Mathematics “Simion Stoilow” of the Romanian Academy, P.O. Box 1-764, RO 014700, Bucharest, Romania}
\email{marian.aprodu@fmi.unibuc.ro, marian.aprodu@imar.ro}
\thanks{MA was partly supported by the PNRR grant
	CF 44/14.11.2022 \emph{Cohomological Hall algebras of smooth surfaces and
		applications}}
\author{Andrea Bruno}
\address{Dipartimento di Matematica e Fisica, Universit\`a Roma Tre 
	\\ L.go S.L. Murialdo, 1 -  00146 Roma, Italia}
\email{andrea.bruno@uniroma3.it}
\thanks{AB was partly supported by PRIN2020-2020KKWT53 and PRIN2022-2022L34E7W}
\author{Edoardo Sernesi}
\address{Dipartimento di Matematica e Fisica, Universit\`a Roma Tre 
	\\ L.go S.L. Murialdo, 1 -  00146 Roma, Italia}
\email{sernesi@gmail.com}

\begin{document}
	
	\maketitle

	\begin{abstract}
		The present paper is a natural continuation of the previous work \cite{ABS19} where we studied the second syzygy scheme of canonical curves. We find sufficient conditions ensuring that the second syzygy scheme of a genus--$g$ curve of degree at least $2g+2$ coincide with the curve. If the property $(N_2)$ is satisfied, the equality is ensured by a more general fact emphasized in \cite{ABS19}. If $(N_2)$ fails, then the analysis uses the known case of canonical curves.
	\end{abstract}
	
	\section{Introduction}
	
	Let $V$ be an $(N+1)$--dimensional complex vector space and fix a basis $x_0,\ldots,x_N$ in $V$, and consider $p\ge 1$. M. Green \cite{Gr82} associates to any non--zero element 
	\[
	\wedge^pV\otimes V\ni \gamma=\sum _{i_1<\ldots<i_p}(x_{i_1}\wedge\cdots\wedge x_{i_p})\otimes\ell_{i_1\ldots i_p}
	\]
	a number of quadratic forms in the following natural way. The image $\delta(\gamma)$ of $\gamma$ through the Koszul differential $\delta:\wedge^pV\otimes V\to \wedge^{p-1}V\otimes \Sym^2V$ is en element of type
	\[
	\delta(\gamma)=\sum_{j_1<\ldots<j_{p-1}} x_{j_1}\wedge\ldots\wedge x_{j_{p-1}}\otimes Q_{j_1\ldots j_{p-1}} 
	\]
	with the property that
	\[
	\sum_{j_1<\ldots<j_{p-1}}\sum_{k}(-1)^k(x_{j_1}\wedge\cdots\wedge \widehat{x}_{j_k}\wedge\cdots\wedge x_{j_{p-1}})\otimes (x_{j_k}\cdot Q_{j_1\ldots j_{p-1}})=0,
	\]
	and the quadratic forms in question are precisely $(Q_{j_1\ldots j_{p-1}})_{j_1<\ldots<j_{p-1}}$.
	
	\medskip
	
	The linear space spanned by them in $\Sym^2V$ is called the space of quadratic forms \emph{involved in $\gamma$}. Furthermore, these quadratic forms generate a homogeneous ideal  $\bI(\gamma)$ of $\bC[x_0,\ldots,x_N]$ called the \emph{syzygy ideal of $\gamma$}, and the scheme $\Syz(\gamma)$ defined by this ideal is the \emph{syzygy scheme of $\gamma$}. If these quadratic forms all vanish on a given projective scheme $X\subset\bP^N$ i.e. $\bI(\gamma)\subset I_X$, we say that $\gamma$ is a \emph{linear syzygy of $X$}, \cite{mG84}. Note that $\Syz(\gamma)$ is the largest scheme for which $\gamma$ is a linear syzygy. For the fixed scheme $X$, the \emph{p--th syzygy scheme} $\Syz_p(X)$ is ideal--theoretically defined by all the quadrics involved in syzygies of $X$ (see \cite{Gr82}, \cite{Ehb}, \cite{vB} etc.).
	
	\medskip
	
	If larger than the given scheme $X$, syzygy schemes tend to have small degrees, and hence they make an impact on the geometry of $X$, provided that the space of syzygies is non--trivial, of course, see \cite{mG84}, \cite{vB}, \cite{Ehb}, \cite{AN} etc. It is thus natural to ask whether we can find explicit sufficient conditions ensuring that the a given syzygy scheme actually coincides with $X$. We focus on the case $p=2$, which presents several advantages, the first and foremost being the simple form of the families of quadrics involved in syzygies. More specifically, for a given linear syzygy $\gamma$, the element $\delta(\gamma)$ is of the form $\delta(\gamma)=\sum_{j=0}^N x_j\otimes Q_j$ and has the property that $\sum_{j=0}^N x_j\cdot Q_j=0$.
	
	\medskip
	
	A first result in this direction is Proposition 2 in \cite{ABS19}: if $X$ satisfies $(N_2)$, then $\Syz_2(X)=X$, see also Proposition \ref{prop:N2} below. On the other hand, if $(N_1)$ fails, then $\Syz_2(X)\ne X$, by definition. Hence, the non--trivial cases to be analyzed are the cases when $X$ satisfies $(N_1)$ i.e. it is projectively normal and the ideal is generated by quadratic forms. In \cite{ABS19} we analyzed this problem for (non--hyperelliptic) canonical curves. The main result in loc.cit. shows that the second syzygy scheme of a canonical curve coincides with the curve, unless it is bielliptic, trigonal, or is contained in a del Pezzo surface.
	
	\medskip
	
	In this paper, we continue the study of the second syzygy schemes. In Section \ref{sec:prelim} we recall all the basic notions and we give criteria for the equality $\Syz_2(X)=X$, Lemma \ref{lem:Syz2(X)=X}, Lemma \ref{lem:Syz2-YZ}. In the opposite direction, we also give a criterion for the non--equality $\Syz_2(X)\ne X$, Lemma \ref{lem:Syz2-notX}. Finally, in this section we prove a semi--continuity result for the syzygy schemes, see Theorem \ref{thm:semicontinuity}.
	Having established the basics, we focus on the case of curves of large degree. The main result is the following.
	
	\begin{thm}
		Let $C$ be a smooth projective curve of genus $\ge 11$ embedded by a line bundle $L$ of degree $d\ge 2g+2$. Suppose that either: 
		\begin{itemize}
			\item[(i)] $d\ge 2g+3$, or 
			\item[(ii)] $d=2g+2$ and there are no 4--secant 2--planes to $C$, or 
			\item[(iii)] $d=2g+2$, $L$ embeds $C$ with a 4--secant 2--plane i.e. $L=K_C+D$ with $D\in C_4$,  $\mathrm{Cliff}(C)\ge 2$ and $C$ is not bielliptic, or
			\item[(iv)] $d=2g+2$, $C$ is bielliptic and $L=K_C+D$ with $D\in C_4$ general. 
		\end{itemize}
		Then $\Syz_2(C)=C$.
	\end{thm}

	The first two cases (i) and (ii) are covered by Proposition \ref{prop:N2}, and are discussed in Section~\ref{sec:non4-secant}. Sections 4 and 5 are devoted to the proof of the case (iii) and represent the technical core of the paper. In Section 4, we reduce the case to the analysis of the second syzygy variety for the variety $X$ that has two irreducible components, namely, the curve $C$ and the 4--secant 2--plane, see Figure \ref{fig:X}. In Section 5, we prove that this reducible variety $X$ satisfies the hypotheses of Lemma \ref{lem:Syz2(X)=X}. We conclude the paper with the case (iv), see Section \ref{sec:excepted}. In the bielliptic case, we use semicontinuity, and specialize to the case when $|D|$ is a $g^1_4$.
	
	\section{Preliminaries}
	\label{sec:prelim}
	
	\subsection{Syzygy schemes}
	
	We start from the general definition of \emph{syzygy schemes}, see \cite{Gr82}, \cite{mG84}, \cite{Ehb}, \cite{vB}, \cite{AN}, \cite{Se}, \cite{ES}. For a projective variety $X$, we denote by $K_{p,q}(X)$ the Koszul cohomology groups of its homogeneous coordinate ring, \cite{mG84}. We refer to \cite{Gr82}, \cite{mG84}, \cite{AN} for some notation, for definitions and basic properties of Koszul cohomology and syzygies. 
	
	\medskip
	
	We fix $V$ a (complex) vector space of dimension $(N+1)$, with $N\ge 3$. For any integer $p$ we consider the Koszul differential (cf. \cite{mG84})
	\[
	\delta_{p,q}:\bigwedge^pV\otimes\Sym^qV\to \bigwedge^{p-1}V\otimes\Sym^{q+1}V
	\]
	given by
	\[
	\delta_{p,q}((v_1\wedge\cdots\wedge v_p)\otimes P=\sum_k(-1)^k(v_i\wedge\cdots\wedge \widehat{v}_k\wedge\cdots\wedge v_p)\otimes (v_kP)
	\]
	for any $v_1\wedge\cdots\wedge v_p\ne 0$.
	
	To ease notation, we will drop the indices $p,q$ whenever they are clear from the context.
	
	\begin{definition}[see \cite{Gr82}, \cite{Ehb}, \cite{vB}, \cite{AN}]
		Let $p\ge 1$ be an integer, and $\gamma\in \wedge^pV\otimes V$. Identify $\delta(\gamma)\in \wedge^{p-1}V\otimes \Sym^2V$ with the linear map $\overline{\gamma}\in\Hom(\wedge^{p-1}V^\vee,\Sym^2V)$. 
		\begin{itemize}
			\item[(i)] 	The \emph{ideal of $\gamma$} is the homogeneous ideal $\mathbf{I}(\gamma)$ of $\Sym(V)$ generated by the image of $\overline{\gamma}$.
			
			\item[(ii)] The \emph{syzygy scheme of $\gamma$} is the projective scheme $\Syz(\gamma)$ in $\mathbf{P}V^\vee$ defined by the ideal $\mathbf{I}(\gamma)$.    		
		\end{itemize}
	\end{definition}
	
	We will recurrently use the following terminology, in connection with the analysis of specific sets of quadrics. 
	
	\begin{definition}
		\begin{itemize}
			\item[(i)] A quadratic form $Q\in\Sym^2V$ is said to be \emph{involved in $\gamma$} if $Q$ belongs to the image of $\overline{\gamma}$ in $\Sym^2V$. In particular, the ideal of $\gamma$ is generated by the quadrics involved in $\gamma$.
			
			\item[(ii)] Let $Q_1,\ldots,Q_r\in \Sym^2V$ be a set of non--zero quadratic forms. We say that they are \emph{linearly related} if there exists $\gamma\in\wedge^2V\otimes V$ such that $Q_1,\ldots,Q_r\in\Im(\overline{\gamma})$, equivalently there exist quadratic forms $Q_{r+1},\ldots,Q_s$ and non--zero linear forms  $\ell_1,\ldots,\ell_s$ such that $\sum_{i=1}^s\ell_i\otimes Q_i\ne 0$ in $V\otimes \Sym^2V$ and $\sum_{i=1}^s\ell_i\cdot Q_i = 0$ in $\Sym^3V$.
		\end{itemize}
	\end{definition}

	Note that, by the properties of the Koszul complex the definition above only depends on $\delta(\gamma)$ and hence on the class $[\gamma]$ of $\gamma$ in the quotient space $\wedge^pV\otimes V/ \wedge^{p+1}V$. If $\delta(\gamma)=0$, then obviously $\Syz(\gamma)=\mathbf{P}V^\vee$, and in general, the syzygy scheme $\Syz(\gamma)$ is the largest scheme in $\mathbf{P}V^\vee$ for which $\gamma$ is a syzygy, i.e. $\gamma\in K_{p,1}(\Syz(\gamma))$.
	
	\medskip
	
	It is sometimes convenient in the definition above to pass to the projectivization, and to this end we define the linear subspace 
	\[
	\mathbf{P}(\gamma):=\mathbf{P}(\Im(\overline{\gamma}))\subset\mathbf{P}(\Sym^2V).
	\] 
	It parameterizes all the quadrics in $\mathbf{P}V^\vee$ that contain $\Syz(\gamma)$. 
	
	\medskip
	
	We now pass to a more specific situation, where the syzygy schemes are related to a given scheme.
	
	\begin{definition}[see \cite{Gr82}, \cite{Ehb}, \cite{vB}, \cite{AN}]
		Let $X\subset\mathbf{P}V^\vee$ be a non-degenerate scheme. \emph{The $p$--th syzygy scheme of $X$} is the scheme--theoretic intersection
		\[
		\Syz_p(X):=\bigcap_{0\ne \gamma\in K_{p,1}(X)}\Syz(\gamma).
		\]
	\end{definition}

	A better way to view a syzygy scheme is the following. Note that we have an isomorphism
	\[
	K_{p,1}(X)\cong\ker\left\{\bigwedge^{p-1}V\otimes I_{X,2}\to \bigwedge^{p-2}V\otimes I_{X,3}\right\}
	\]
	via $\gamma\mapsto \delta(\gamma)$. The inclusion of $K_{p,1}(X)$ into $\bigwedge^{p-1}V\otimes I_{X,2}$ yields to a natural map
	\begin{equation}
		\label{eqn:phi}
		\phi:K_{p,1}(X)\otimes \bigwedge^{p-1}V^\vee\to I_{X,2},
	\end{equation}
	and then
	\[
	\Syz_p(X)=\bV(\Im(\phi)),
	\]
	since
	\[
	\Im(\phi)=\sum_{0\ne\gamma\in K_{2,1}(X)} \Im(\overline{\gamma}).
	\]
	
	We call the ideal generated by $\Im(\phi)$ above, the \emph{$p$--the syzygy ideal of $X$}.
	
	\medskip
	
	In some of the previous works on this topic, the syzygy schemes were considered automatically to be endowed with the reduced structure, whence the terminology \emph{syzygy varieties}. In the present paper, we will show that the syzygy schemes that we analyze are automatically reduced (and hence they are genuine varieties) due to some specific reasons related to the particular setup. We also notice that the scheme structure of the syzygy schemes that appear in our paper \cite{ABS19} are reduced, for similar reasons. The general criteria that we apply (and have been implicitly applied also in \cite{ABS19}) are listed in Lemmas \ref{lem:Syz2(X)=X}, \ref{lem:Syz2-YZ}, \ref{lem:Syz2-notX} below. However, syzygy schemes need not be reduced in general, see for example Proposition \ref{prop:N2} below.

	\subsection{The second syzygy schemes}
	
	We focus next on the case $p=2$. Since a non--degenerate variety $X$ is contained in any of its associated syzygy varieties, and syzygy varieties are cut out by quadrics, it is natural to ask the following question.
	
	\begin{qn}
		\label{qn:Syz2}
		Assume the ideal of $X$ is generated by quadrics. Under which hypotheses does the equality $\Syz_2(X)=X$ hold?
	\end{qn}
	
	We will partially address this question for curves. 
	
	\medskip
	
	We analyze next this condition in full generality and immediately obtain the following criterion for $\Syz_2(X)$ to be equal to $X$. 
	
	\begin{lemma}
		\label{lem:Syz2(X)=X}
		Notation as above. If the map 
		\[
		\phi:K_{2,1}(X)\otimes V^\vee \to I_{X,2}
		\]
		is surjective, then $\Syz_2(X)=X$.	In other words,  $\Syz_2(X)=X$ if there exists a basis in $I_{X,2}$ formed by quadratic forms that are involved in linear syzygies of $X$.
	\end{lemma}

	\proof
	We apply the definition.
	\endproof
	
	From the projective--geometry view point, the map $\phi$ above is surjective if and only if the union of linear subspaces $\mathbf{P}(\gamma)$ for all non--zero syzygies $\gamma$ of $X$ span the whole space of quadrics that contain $X$, i.e.
	\[
	\left\langle \bigcup_{0\ne\gamma\in K_{2,1}(X)}\mathbf{P}(\gamma)\right\rangle=\mathbf{P}(I_{X,2}).
	\]
	
	In some particular cases, it might be that $\mathbf{P}(\gamma)$ already equals $\mathbf{P}V^\vee$ for a given $\gamma$. For example, take the twisted cubic in $\bP^3$ with homogeneous coordinates $x_0,x_1,x_2,x_3$. The ideal of the twisted cubic is generated by the three quadrics $Q_{01}=x_0x_2-x_1^2$, $Q_{02}=x_0x_3-x_1x_2$ and $Q_{12}=x_1x_3-x_2^2$. For 
	\[
	\gamma=\frac{1}{2}x_0\wedge x_1\otimes x_3-x_0\wedge x_2\otimes x_2+\frac{1}{2}x_0\wedge x_3\otimes x_1+x_1\wedge x_2\otimes x_1-\frac{1}{2}x_1\wedge x_3\otimes x_0
	\] 
	we have $\delta(\gamma)=x_0\otimes Q_{12}-x_1\otimes Q_{02}+x_2\otimes Q_{01}$, and hence $\Syz(\gamma)$ equals the twisted cubic. 
	In other cases, the linear spaces $\mathbf{P}(\gamma)$ might be strictly smaller than the ambient space $\bP(I_{X,2})$, while their union coincides with the whole space. This is the case in the following general situation:

	\begin{prop}[\cite{ABS19}, Prop. 2]
		\label{prop:N2}
		Assume $X$ is projectively normal, of codimension at least 2, and satisfies property $(N_2)$. Then any quadric in the ideal of $X$ is involved in some non--trivial linear syzygy of $X$ i.e.
		\[
		\bigcup_{0\ne\gamma\in K_{2,1}(X)}\mathbf{P}(\gamma)=\mathbf{P}(I_{X,2}).
		\]
		In particular, $\Syz_2(X)=X$.
	\end{prop}
	
	\proof
	For convenience, we repeat here the argument from \cite{ABS19} given for varieties. By hypothesis, the ideal of $X$ is generated by (at least two) quadratic forms, and hence it suffices to show that every $Q\in I_{X,2}$  is involved in a linear syzygy of $X$. 
	
	Assume that a quadratic form $Q$ is not involved in any linear syzygy of $X$, and choose a basis $Q=Q_1,\ldots,Q_m$ of $I_{X,2}$.  A basis for the space of linear syzygies is formed by rows of linear foms $R_i=(\ell_{i1},\ldots,\ell_{im})$ with $\sum\ell_{ij}\cdot Q_j=0$ for all $i$. Since $Q_1$ is not involved in any linear syzygy, we must have $\ell_{i1}=0$ for all $i$. The quadratic syzygy $-Q_2\cdot Q_1+Q_1\cdot Q_2=0$
	is represented by the row $(-Q_2,Q_1,0,\ldots,0)$, and hence it is a combination of rows $R_i$ with linear forms as coefficients. We obtain $Q_2=0$, which is absurd. 
	\endproof

	This Proposition applies, for example, to rational normal curves of degree at least 4. Note also that, if $X$ is non--reduced, this Proposition produces non--reduced syzygy schemes, for example, if $X$ is a ribbon, see \cite{BE}, \cite{RS}.
	
	\medskip
	
	If the property $(N_2)$ fails, the general strategy to verify that $\Syz_2(X)=X$ is to find a basis of quadratic forms in the ideal of $X$ which are all involved in syzygies of $X$.
	In this direction, the following result is quite useful, and will be applied in the next section.

	\begin{lemma}
		\label{lem:Syz2-YZ}
		Let $Y \subset Z \subset \bP V^\vee$ with the property $(N_1)$, and assume that, scheme--theoretically
		\[
		Y = Z \cap V(Q_1) \cap \ldots \cap V(Q_m)
		\]
		with $Q_i \in I_{Y,2}$ such that each $Q_i$ is involved in a linear syzygy of $Y$. 
		Assume that the map $K_{2,1}(Z)\otimes V^\vee\to I_{Z,2}$ is surjective (which implies $\mathrm{Syz}_2(Z)=Z$). Then the map $K_{2,1}(Y)\otimes V^\vee\to I_{Y,2}$ is also surjective, and hence 
		we also have the equality $\mathrm{Syz}_2(Y) =Y$.
	\end{lemma}
	
	\begin{proof}
		As observed before, the surjectivity of the map is equivalent to the existence of a basis in $I_{Y,2}$ such that any element of the basis in involved in a linear syzygy of $Y$. By the hypothesis, there exists a basis in $I_{Z,2}$ such that any of these quadratic forms is involved in a syzygy of $Z$, and hence also involved in a syzygy of $Y$. On the other hand, we have $I_{Y,2}=I_{Z,2}+\Span(Q_1,\ldots,Q_r)$, and $Q_1,\ldots,Q_r$ are involved in syzygies of $Y$. Hence we can find a basis in $I_{Y,2}$ composed only by quadratic forms involved in linear syzygies of~$Y$.
	\end{proof}
	
	On the negative side, we have the following result.
	
	\begin{lemma}
		\label{lem:Syz2-notX}
		Let $Z \subset \mathbf{P}^N$ with the property $(N_2)$, and $Y\subset Z$ with the property $(N_1)$ and assume that, scheme--theoretically
		\[
		Y = Z \cap V(Q)
		\]
		with $Q \in I_{Y,2}$ such that $Q$ is not involved in any linear syzygy of $Y$. Then 
		\[
		\mathrm{Syz}_2(Y) =Z.
		\]
	\end{lemma}
	
	\proof
	Note that, by hypothesis, $I_{Y,2}=I_{Z,2}+\langle Q\rangle$ and $\Syz_2(Z)=Z$. 
	
	We claim that given any quadratic form $Q'\in I_{Y,2}\setminus I_{Z,2}$, we have that $Q'$ is not involved in any linear syzygy of $Y$. Indeed, assume this is not the case, then there exist quadratic forms $Q'_1,\ldots,Q'_r\in I_{Y,2}$ and linear forms $\ell',\ell'_1,\ldots,\ell'_r$ such that $\ell'\otimes Q'+\sum  \ell'_i\otimes Q'_i\ne 0$ and
	\begin{equation}
		\label{eq:Syz}
		\ell'\cdot  Q'+\sum \ell'_i\cdot  Q'_i=0.
	\end{equation} 
	Upon rescaling simultaneously the quadratic forms and the linear forms involved in the given relation, we may assume $Q'=Q+Q''$, $Q'_i=Q+Q''_i$, with $Q'', Q_i''\in I_{Z,2}$. Replacing $Q',Q'_i$ in (\ref{eq:Syz}) we find that $Q$ is involved in a linear syzygy, contradiction.

	Hence, the only quadratic forms in $I_{Y,2}$ involved in a syzygy are those of $I_{Z,2}$, closing the proof.
	\endproof
	
	%
	%
	%
	%
	%

	\subsection{Semi--continuity of syzygy ideals}
	We prove next a semicontinuity result for syzygy ideals.
	
	\begin{thm}
		\label{thm:semicontinuity}
		Suppose $\cX\in S\times \bP V^\vee$ is a flat family of non--degenerated projective schemes with property $(N_1)$, parameterized by a scheme $S$ such that the  graded Betti numbers of the fibres $\cX_s$ of the projection to $S$ are constant in the family; in particular, the Hilbert function of $\cX_s$ is independent on $s$. Then the function
		\[
		S\ni s\mapsto \mathrm{dim}
		\left(\Im\{K_{2,1}(\cX_s)\otimes V^\vee\stackrel{\phi_s}{\longrightarrow} I_{\cX_s,2}\}\right)
		\]
		is lower--semicontinuous. In particular, the locus 
		\[
		\{s\in S|\ \phi_s\mbox{ is surjective}\}
		\]
		is open.
	\end{thm}
	
	\proof
	From \cite{AN}, Prop. 1.31 and the hypothesis, there exists a vector bundle $\cK_{2,1}(\cX/S)$ over $S$ whose fibres are $\cK_{2,1}(\cX/S)\otimes k(s)\cong K_{2,1}(\cX_s)$ for all $s\in S$ (in this case, the open subset $U$ mentiones in loc.cit. coincides with $S$). Since the dimension of $I_{\cX,2}$ is independent of $s$ as well, there exists another vector bundle $\mathcal J$ on $S$ with fibres $\mathcal{J}\otimes k(s)=I_{\cX,2}$ for all $s\in S$. The bundle $\mathcal{J}$ is the kernel of the natural vector bundle morphism 
	\[
	\cO_S\otimes \Sym^2V=p_{1,*}(\cO_S\boxtimes \cO_{\bP V^\vee}(2))\to p_{1,*}((\cO_S\boxtimes \cO_{\bP V^\vee}(2))|_\cX).
	\]
	
	The collection of linear maps $(\phi_s)_{s\in S}$ globalize to a vector bundle morphism 
	\[
	\Phi:\cK_{2,1}(\cX/S)\otimes V^\vee\to \mathcal{J},
	\]
	and the general theory of determinantal loci concludes the proof. 
	\endproof
	
	Note that working in the case $p=2$ is not essential here, and a similar statement with identical conclusions holds for any $p\ge 2$.
	

	\subsection{The second syzygy schemes of canonical curves}
	In our previous work \cite{ABS19}, we answered Question \ref{qn:Syz2} for canonical curves. For convenience, taking into account that the proof of the main result of the present paper goes by reduction to the canonical case, we recall here the strategy of \cite{ABS19}. As already mentioned in Proposition \ref{prop:N2} we proved first that for an arbitrary projective variety $X$, the property $(N_2)$ implies that $\Syz_2(X)$ is equal to $X$. Taking into account that the property $(N_2)$ is satisfied for a canonical curve $C$ if the Clifford index is at least three, \cite{Vo} and \cite{fS91}, we can limit to considering the case of curves of Clifford index two. Plane sextics were treated in \cite{ABS19}, Section 3, and we found that the second syzygy scheme is the Veronese surface that contains the curve; to conclude, we tacitly apply Lemma \ref{lem:Syz2-notX}. From the same Lemma \ref{lem:Syz2-notX}, we infer that the second syzygy scheme of a curve of genus between 5 and 10 that is contained in a del Pezzo surface coincides with the del Pezzo surface, see \cite{ABS19}, Section 4. Similarly, the second syzygy scheme of a bielliptic curve is a cone over an elliptic curve.
	The next case is that of canonical tetragonal curves non--bielliptic of genus at least 11. They are naturally contained in a three--dimensional scroll and we analyze their equations in the scroll, closely following \cite{fS86}. Using the rolling--factors trick, in loc.cit. we prove the following:
	
	\begin{thm}
		\label{thm:ABS}
		Let $C$ be a non--bielliptic tetragonal canonical curve of genus $\ge 11$. Then the natural map:
		\[
		K_{2,1}(C)\otimes H^0(C,K_C)^\vee\to I_{C,2}
		\]
		is surjective.
	\end{thm}
	
	In particular, we find that the second syzygy scheme coincides with the curve itself. 
	Note that in \cite{ABS19} we have not been particularly concerned with the scheme structure of $\Syz_2(C)$ for $C$ a canonical curve, however, the comments above show that everything proved in loc.cit. is valid at the scheme structure level.

	\medskip
	
	In the subsequent sections, we analyze the case of curves of degree at least $2g+2$, and their study is subsequently reduced to the canonical embedding where the results of \cite{ABS19} apply.


	\section{Curves with property $(N_2)$}
	\label{sec:non4-secant}
	
	Let $C$  be a projective nonsingular curve of genus $g$ and $L$ a very ample invertible sheaf of degree $d$ on $C$.  We will need the following result:
	
	\begin{thm}[Green--Lazarsfeld]
		\label{T:GL}
		Notation as above.
		\begin{itemize}
			\item[(i)]  If $d \ge 2g+1+p$  then $L$ satisfies $(N_p)$.
			\item[(ii)] If $d=2g+p$ then $L$ satisfies $(N_p)$ unless $C$ is hyperelliptic or has a $(p+2)$-secant $p$-plane.
		\end{itemize}
		
	\end{thm}
	
	Part (i) is in \cite{mG84}, Thm. 4.a.1. Part (ii) is stated in \cite{GL86}, 3.3, with proof postponed to future publication. It is proved in \cite{GL88}, Thm. 2.  
	
	\medskip
	
	Given a projective scheme $X \subset \bP^r$ we let $\mathrm{Syz}_2(X)$ denote the second syzygy scheme of $X$. We will write $\mathrm{Syz}_2(C)$ after identifying $C=\varphi_L(C)\subset \bP H^0(L)^\vee$. We have the following preliminary result:
	
	\begin{prop}
		Assume that $d \ge 2g+3$. Then $\mathrm{Syz}_2(C)=C$.
	\end{prop}
	
	\begin{proof}
		By Theorem \ref{T:GL}(i), $L$ satisfies $(N_2)$. Then apply   
		Proposition \ref{prop:N2}.
	\end{proof}
	
	The proposition shows that, for what concerns $\mathrm{Syz}_2(C)$,  the only interesting cases arise for $d \le 2g+2$. We will start considering the case $d=2g+2$.

	\begin{prop}
		\label{prop:non4-secant}
		Assume that $d = 2g+2$, and $L-K_C$ is not effective. Then $\mathrm{Syz}_2(C)=C$.
	\end{prop}
	
	\begin{proof}
		By Theorem \ref{T:GL}(ii), $L$ satisfies $(N_2)$. Then apply Proposition \ref{prop:N2}.
	\end{proof}

	\begin{rem}
		\label{rmk:hyperelliptic}
		In the hyperelliptic case, the curve satisfies $(N_1)$ and fails $(N_2)$, by Green--Lazarsfeld's result. In this case, the surface $S$ spanned by the chords of $C$ defined by the divisors of the hyperelliptic pencil is a scroll, and hence it is a surface of (minimal) degree $g+1$.  Since
		$$
		h^0(\bP, \c I_C(2))=h^0(\bP,\c I_S(2))+1 
		$$
		we see that $C = S \cap Q$, where $Q\in I_{C,2}\setminus I_{S,2}$.  Moreover, since the ideal of $C$ is generated by quadratic forms, we obtain the exact sequence
		\[
		0\to K_{2,1}(C)\to I_{C,2}\otimes V\to I_{C,3}\to 0
		\]
		and hence, taking into account that $C$ is projectively normal, we compute 
		\[
		\mathrm{dim}(K_{2,1}(C))=\frac{g(g^2-1)}{2}=\mathrm{dim}(K_{2,1}(S)),
		\]
		which implies that $Q$ cannot be involved in any linear syzygy.
		Therefore, from Lemma \ref{lem:Syz2-notX}, it follows that
		$$
		\mathrm{Syz}_2(C) = \mathrm{Syz}_2(S)=S.
		$$
		
		In contrast to this situation, in the cases analyzed in Sections \ref{sec:4-secant} and \ref{sec:excepted}, we will find that $\Syz_2(C)=C$.
	\end{rem}
	
	
	\section{Non--bielliptic curves of degree $2g+2$  with a 4--secant 2--plane}
	\label{sec:4-secant}
	
	We analyze in this section the following situation. Consider $C$ a smooth curve projective of genus $g \ge 11$, and Cliff$(C) \ge 2$ embedded in $\bP:=\bP^{g+2}$ by a line bundle $L$ of degree $d=2g+2$. Contrary to Proposition \ref{prop:non4-secant}, we assume next that the embedding has a 4--secant 2--plane $\Pi$, i.e. $L=K_C+D$ where $D=C\cap \Pi$ is an effective divisor of degree four. Denote $V=H^0(C,L)$. The 2--plane $\Pi$ is defined as $\Pi=\bP W^\vee$, where $W$ is given by the exact sequence
	\[
	0\to H^0(C,K_C)\to V\to W\to 0.
	\]

	Put $X=C\cup \Pi$.
	
	\medskip

	\begin{figure}[h!]
		\centering
		\begin{tikzpicture}
			
			
			\draw [red, thick, xshift=4cm] (5.6,-0.2) node {$C$};
			\draw [xshift=4cm] (4.3,-0.6) node {$\Pi$};

			\draw [gray!50, fill=gray!10, xshift=4cm]  (-1,0) -- (3,-1) -- (6,0.5) -- (2.5,1)  -- cycle;
			
			\draw [red, thick, xshift=4cm] (0.3,0) node {$x_1$};
			\draw [red, thick, xshift=4cm] (1.8,0.4) node {$x_2$};
			\draw [red, thick, xshift=4cm] (3.3,0) node {$x_3$};
			\draw [red, thick, xshift=4cm] (4.8,0.4) node {$x_4$};
			
			\draw [red, thick, xshift=4cm] plot [smooth] coordinates {(0,0) (1,1) (1.6,1) (2,0.5)};
			\draw [red, dashed, xshift=4cm] plot [smooth] coordinates {(0,0) (-0.17,-0.2)};
			\draw [red, thick, xshift=4cm] plot [smooth] coordinates {(-0.17,-0.2) (-0.45,-0.5)};
			
			\draw [red, thick, xshift=4cm] plot [smooth] coordinates {(3,0) (4,1) (4.6,1) (5,0.5)};
			\draw [red, dashed, xshift=4cm] plot [smooth] coordinates {(2,0.5) (2.5,-0.25) (3,0)};
			
			\draw [red, dashed, xshift=4cm] plot [smooth] coordinates {(5,0.5) (5.2,0.1)};
			\draw [red, thick, xshift=4cm] plot [smooth] coordinates {(5.2,0.1) (5.3,-0.2)};
			
			\filldraw [red, thick, xshift=4cm] (0,0) circle (0.5pt); 
			\filldraw [red, thick, xshift=4cm] (2,0.5) circle (0.5pt); 
			\filldraw [red, thick, xshift=4cm] (3,0) circle (0.5pt); 
			\filldraw [red, thick, xshift=4cm] (5,0.5) circle (0.5pt); 
			
			
		\end{tikzpicture}
		\caption{The variety $X$ when $D=x_1+x_2+x_3+x_4$ is reduced.} \label{fig:X}
	\end{figure}
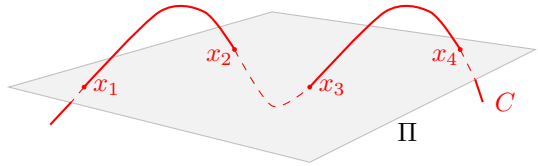

	\medskip
	
	We will use the following easy Lemma.
	
	\begin{lemma}
		\label{lem:4PointsInPi}
		We have $h^0(\Pi,\c I_{D/\Pi}(1))=0$, $h^0(\Pi,\c I_{D/\Pi}(2))=2$ and $h^1(\Pi,\c I_{D/\Pi}(2))=0$.
	\end{lemma}
	
	\proof
	From the exact sequence
	\[
	0\to \c I_{C/\bP}\to \c I_{D/\bP}\to \c I_{D/C}\to 0
	\]
	we obtain an isomorphism $H^0(\bP,\c I_{D/\bP}(1)\cong H^0(C,\c I_{D/C}(1))=H^0(C,K_C)$. The exact sequence
	\[
	0\to H^0(\bP,\c I_{\Pi/\bP}(1))\to H^0(\bP,\c I_{D/\bP}(1))\to H^0(\Pi,\c I_{D/\Pi}(1))\to 0
	\]
	implies then that $h^0(\Pi,\c I_{D/\Pi}(1))=0$.
	
	\medskip
	
	From the diagram (using the fact that $h^1(\bP,\c I_{C/\bP}(2))=0$)
	\[
	\xymatrix@R=0.7cm@C=0.5cm{
		&   0 \ar[d] & 0 \ar[d] &&\\
		0 \ar[r] &  H^0(\bP,\c I_{C/\bP}(2))\ar[r]\ar[d] & H^0(\bP,\c I_{D/\bP}(2))  \ar[r]\ar[d] & \ar[r] H^0(C,\c I_{D/C}(2)) & 0\\
		& H^0(\bP,\cO_\bP(2))\ar@{=}[r]\ar[d] & H^0(\bP,\cO_\bP(2)) \ar[d] & & \\
		& H^0(C,\cO_C(2)) \ar[d]\ar[r] & H^0(D,\cO_D) &&\\
		& 0 &&&
	}
	\]
	applying the Snake Lemma, 
	we obtain the surjectivity of the map $\Sym^2V\to H^0(D,\cO_D)$, which implies that $h^0(\bP,\c I_{D/\bP}(2))={g+4\choose 2}-4$ and $h^1(\bP,\c I_{D/\bP}(2))=0$. 
	
	To conclude, use the exact sequence
	\[
	0\to \c I_{\Pi/\bP}(2)\to \c I_{D/\bP}(2)\to \c I_{D/\Pi}(2)\to 0.
	\]
	and the fact that $h^0(\bP,\c I_{\Pi/\bP}(2))={g+4\choose 2}-6$ and $h^1(\bP,\c I_{\Pi/\bP}(2))=0$.
	\endproof

	We state now the main result of the paper:
	
	\begin{thm}
		\label{thm:main}
		Notation as above. If $C$ is not bielliptic, then $\mathrm{Syz}_2(C)=C$.   
	\end{thm}

	\begin{proof}
		The line bundle $L$ embeds $C \subset \bP^{g+2}=:\bP$ as a curve of degree $2g+2$ and satisfies $(N_1)$ by Theorem \ref{T:GL}(i) for $p=1$.  Therefore
		\[
		h^0(\bP, \c I_{C/\bP}(2))= \binom{g+1}{2}+1
		\]

		{\bf Claim 1:}  $H^1(\c I_{X/\bP}(k))=0$ for all $k \ge 1$, i.e. $\cO_X(1)$  satisfies $(N_0)$.
		
		\medskip
		
		Using the exact sequence
		\[
		\xymatrix{
			0 \longrightarrow \c I_{X/\bP}(k) \longrightarrow \c I_{C/\bP}(k) \ar[r]^-{\beta_k} & \c I_{D/\Pi}(k)\longrightarrow 0}
		\]
		and the fact that $H^1(\c I_{C/\bP}(k))=0$ for all $k\ge 0$, it suffices to prove that  $H^0(\beta_k)$ is surjective for all $k \ge 0$.  This is trivially true for $k=0$, and it is true for $k=1$ from Lemma \ref{lem:4PointsInPi}.
		For $k=2$, we have $h^0(\c I_{D/\Pi}(2))=2$, from Lemma \ref{lem:4PointsInPi}. If $H^0(\beta_2)=0$, then any quadratic form in $I_{C,2}$ vanishes on $\Pi$ which is in contradiction with the fact that the ideal of $C$ is generated in degree two. Similarly, if the image of $H^0(\beta_2)$ is one--dimensional and hence it defines a conic in $\Pi$, any quadratic form in $I_{C,2}$ would vanish on this conic. Hence $H^0(\beta_2)$ is surjective. If $k\ge 3$ we proceed by induction on $k$ using the following commutative diagram and the fact that both $\mathcal I_{C/\mathbb P}$ and $\mathcal I_{D/\Pi}$ are generated by quadrics: 
		\begin{equation}\label{E:exseqX}
			\xymatrix@C=2cm{
				H^0(\c I_{C/\bP}(k))\ar[r]^-{H^0(\beta_k)}&H^0(\c I_{D/\Pi}(k)\\
				H^0(\c I_{C/\bP}(k-1))\otimes V\ar[r]^{H^0(\beta_{k-1})\otimes \mathrm{id}_V}\ar[u]& H^0(\c I_{D/\Pi}(k-1))\otimes V\ar[u]
			}
		\end{equation}
		
		It follows also that:
		\[
		h^0(\bP,\c I_{X/\bP}(2))=\binom{g+1}{2}-1=h^0(\bP,\c I_{C/\bP}(2))-2,
		\]
		in particular, $I_{C,2}$ contains two extra linearly independent quadratic forms that do not belong to $I_{X,2}$.
		
		\medskip
		
		{\bf Claim 2:} $\cO_X(1)$ satisfies $(N_1)$.
		
		\medskip
		
		Consider the following diagram obtained using the equality $\c I_{\Pi/X}=\c I_{D/C}$. 
		\[
		\xymatrix@C=0.5cm{
			&0\ar[d]&0\ar[d]&0\ar[d]\\
			0\ar[r]&K_{2,1}(X)\ar[r]\ar[d]&\ker(\mu_\Pi)\ar[d]\ar[r]^-\varphi &\ker(\delta)\ar[d]\\
			0\ar[r]& H^0(\c I_{X/\bP}(2)) \otimes V \ar[d]_-{\mu_X}\ar[r]&H^0(\c I_{\Pi/\bP}(2)) \otimes V \ar[d]_-{\mu_\Pi}\ar[r]&H^0(C, 2K_C+D) \otimes V \ar[d]_-\delta\ar[r]&0\\
			0\ar[r]&H^0(\c I_{X/\bP}(3))\ar[d]\ar[r]&H^0(\c I_{\Pi/\bP}(3))\ar[d]\ar[r]&H^0(C,3K_C+2D)\ar[r]\ar[d]&0\\
			&K_{1,2}(X)&0&0
		}
		\]
		The vanising of $K_{1,2}(X)$ is equivalent to the surjectivity of the map $\varphi:\ker(\mu_\Pi)\to \ker(\delta)$.
		
		The last two vertical sequences arise from a morphism between Koszul complexes:
		$$
		\xymatrix{
			H^0(\c I_{\Pi/\bP}(1))\otimes \wedge^2V\ar[r]\ar[d]&H^0(\c I_{\Pi/\bP}(2))\otimes V\ar[r]\ar[d]&H^0(\c I_{\Pi/\bP}(3))\ar[d]\\
			H^0(C,K_C)\otimes\wedge^2V\ar[r]& H^0(C,2K_C+D) \otimes V\ar[r]^\delta & H^0(C,3K_C+2D)
		}
		$$

		Note that  the cohomology at the middle of the Koszul complex
		\[
		H^0(C,K_C)\otimes\wedge^2V\to H^0(C,2K_C+D) \otimes V\stackrel{\delta}{\to} H^0(C,3K_C+2D)
		\]
		equals $K_{1,1}(C,K_C,L)$. By duality \cite{mG84}, it is isomorphic to $K_{g,1}(C,L)^\vee$ and, since the curve is non--hyperelliptic, the latter is zero, by M. Green's $K_{p,1}$--Theorem \cite{mG84}. Therefore, 
		\[
		\ker(\delta)=\mathrm{Im}\left(H^0(C,K_C)\otimes\wedge^2V\to H^0(C,2K_C+D) \otimes V\right)
		\]
		Also note that the vertical map $H^0(\c I_{\Pi/\bP}(1))\otimes \wedge^2V\to H^0(C,K_C)\otimes\wedge^2V$ is an isomorphism. In conclusion, the induced map from
		$
		\mathrm{Im}(H^0(\c I_{\Pi/\bP}(1))\otimes \wedge^2V\to H^0(\c I_{\Pi/\bP}(2))\otimes V)
		$ 
		to  
		$
		\mathrm{Im}(H^0(C,K_C)\otimes\wedge^2V\to H^0(C,2K_C+D) \otimes V)=\ker(\delta)
		$
		is surjective, and hence  the map  $\varphi$ is surjective as well.

		One can   prove similarly  that $K_{1,q}(X)=0$ for $g \ge 3$.  Claim 2 is proved.
		
		\medskip
		
		{\bf Claim 3:} Any quadratic form $Q\in H^0(\c I_{C/\bP}(2))\setminus H^0(\c I_{X/\bP}(2))$ is  linearly related with $I_{X,2}$.

		\medskip
		
		From Claim 2, we know that $K_{1,2}(X)=0$. Consider the following diagram derived from \eqref{E:exseqX} for $k=3$:
		
		\[
		\xymatrix@C=0.5cm{
			&0\ar[d]&0\ar[d]&0\ar[d]\\
			0\ar[r]&K_{2,1}(X)\ar[r]\ar[d]&K_{2,1}(C)\ar[d]\ar[r]^-\lambda&\ker({\mu_{D/\Pi}})\ar[d]\\
			0\ar[r]&H^0(\c I_{X/\bP}(2))\otimes V \ar[d]_-{\mu_X}\ar[r]&H^0(\c I_{C/\bP}(2))\otimes V \ar[d]_-{\mu_C}\ar[r]&H^0(\c I_{D/\Pi}(2))\otimes V\ar[d]_-{\mu_{D/\Pi}}\ar[r]&0\\
			0\ar[r]&H^0(\c I_{X/\bP}(3))\ar[d]\ar[r]&H^0(\c I_{C/\bP}(3))\ar[d]\ar[r]&H^0(\c I_{D/\Pi} (3))\ar[r]\ar[d]&0\\
			&K_{1,2}(X)=0&0&0
		}
		\]

		Let $0\ne \ell\in H^0(\bP, \c I_{\Pi/\bP}(1))$. From the surjectivity of the multiplication map $H^0(\c I_{X/\bP}(2))\otimes V\to H^0(\c I_{X/\bP}(3))$ it follows that there are quadratic forms $Q_i\in I_{X,2}$ and linear forms $\ell_i$, such that $Q\cdot \ell=\sum Q_i\cdot \ell_i$ which means that $Q$ is linearly related with $I_{X,2}$. 
		
		In particular, if $Q_1,Q_2\in H^0(\c I_{C/\bP}(2))\setminus H^0(\c I_{X/\bP}(2))$ are two linearly independent quadratic forms, then $Q_1,Q_2$  are linearly related with $I_{X,2}$. 
		
		\medskip	
		
		In view of Lemma \ref{lem:Syz2-YZ}, we are reduced to showing the following:
		
		\begin{thm}\label{T:syzX}
			Notation as above. The map $K_{2,1}(X)\otimes V^\vee\to I_{X,2}$ is surjective.
		\end{thm}
		
		Taking Theorem \ref{T:syzX} for granted, the proof of Theorem \ref{T:GL} is now complete.
	\end{proof}
	
	
	\section{Proof of Theorem \ref{T:syzX}}
	
	Let $T$ be the cone over $\varphi_K(C)\subset \bP^{g-1}$ with vertex $\Pi$. We have inclusions:
	\[
	C \subset X \subset T.
	\]
	
	Since $K_{2,1}(\varphi_K(C))\otimes H^0(K_C)^\vee \to I_{\varphi_K(C),2}$ is surjective by Theorem \ref{thm:ABS}, and $I_{\varphi_K(C),2}=I_{T,2}$ we also have the surjectivity of the map:
	\[
	K_{2,1}(T)\otimes V^\vee \to I_{T,2}.
	\]
	
	Therefore, using Lemma \ref{lem:Syz2-YZ} it will suffice to prove the following statement:
	
	\medskip
	
	\noindent
	\emph{$(*)$   Every quadratic form in (a given basis of) the algebraic complement of the subspace $H^0(\c I_{T/\bP}(2))$  in $H^0(\c I_{X/\bP}(2))$ is involved in a linear syzygy of $X$.
	}
	
	\medskip
	
	We record the following useful fact.

	\begin{lemma}
		The exact sequence:
		$$
		\xymatrix{
			0\ar[r]&\c I_{T/\bP}(k)\ar[r]&\c I_{X/\bP}(k)\ar[r]&\c I_{X/T}(k)\ar[r]& 0
		}
		$$
		is exact on global sections for all $k$. In particular:
		$$
		h^0(T,\c I_{X/T}(2))=3g-4= h^0(\c I_{X/\bP}(2)) - h^0(\c I_{T/\bP}(2))
		$$.
	\end{lemma}
	
	\begin{proof}
		The curve $C$ satisfies $(N_1)$ in the canonical embedding, and hence $T\subset\bP$ satisfies $(N_1)$ as well. 
	\end{proof}  
	
	Consider the diagram with exact rows:
	\[
	\xymatrix@C=0.5cm{
		0\ar[r] &	H^0(\c I_{T/\bP}(2))\otimes V\ar[r]\ar[d]^-{\mu_T}&H^0(\c I_{X/\bP}(2))\otimes V\ar[r]\ar[d]^-{\mu_X}&H^0(\c I_{X/T}(2))\otimes V\ar[d]^-{\mu_{X/T}}\ar[r] & 0\\
		0\ar[r] &	H^0(\c I_{T/\bP}(3))\ar[r]&H^0(\c I_{X/\bP}(3))\ar[r]&H^0(\c I_{X/T}(3))\ar[r] &0
	}
	\]
	
	
	Since ${\mu_T}$ is surjective, the map 
	\[
	\ker({\mu_X}) \longrightarrow \ker({\mu_{X/T}})
	\]
	is surjective as well. 
	\medskip
	
	We observe that, in order to prove $(*)$ it suffices to prove that we can find a basis consisting of elements $Q$ in the algebraic complement of $H^0(\c I_{T/\bP}(2))$ in $H^0(\c I_{X/\bP}(2))$, with the property that the image $q\in H^0(\c I_{X/T}(2))$ is involved in a non-trivial element of $\ker({\mu_{X/T}})$.  Indeed, let $q\otimes \ell+\sum q_i\otimes \ell_i$ be a non--zero element of $\ker({\mu_{X/T}})$. For every $i$, choose a lifting $Q_i\in H^0(\c I_{X/\bP}(2))$ of $q_i$. It is immediate that 
	\[
	{\mu_X}(Q\otimes \ell+\sum Q_i\otimes \ell_i)\in H^0(\c I_{T/\bP}(3)).
	\]
	
	Since the Clifford index of the curve equals two, it follows by Petri's Theorem that the multiplication map 
	\[
	H^0(\c I_{C/\bP^{g-1}}(2))\otimes H^0(C,K_C)\to H^0(\c I_{C/\bP^{g-1}}(3))
	\]
	is surjective. On the other hand, the homogeneous ideal of $T$ is the extension of the homogeneous ideal of the canonical curve in the larger polynomial ring, and hence the multiplication map ${\mu_T}$ is surjective, too. In particular, we obtain that there exist elements $Q'_j\in H^0(\c I_{T/\bP}(2))$ and linear forms $\ell'_j$ such that 
	\[
	{\mu_X}\left(Q\otimes \ell+\sum Q_i\otimes \ell_i\right)={\mu_T}\left(\sum Q'_j\otimes \ell'_j\right),
	\]
	i.e. $Q\otimes \ell+\sum Q_i\otimes \ell_i-\sum Q'_j\otimes \ell'_j\in \ker({\mu_X})=K_{2,1}(X)$. On the other hand, this element is non-zero, since it is a lift of the non--zero element $q\otimes \ell+\sum q_i\otimes \ell_i$, which proves that $Q$ is indeed involved in a linear syzygy.
	
	\medskip
	
	In conclusion, in order to complete the proof, we show that every $q\in H^0(\c I_{X/T}(2))$ in a particular basis is involved in a non-trivial element of $\ker({\mu_{X/T}})$. We will find a natural subspace of $H^0(\c I_{X/T}(2))$, and we will prove this claim for elements in this particular subspace, and then for elements in an algebraic complement. To this end, we need some preparations.

	We start with the exact sequences of ideal sheaves:
	\begin{equation}\label{E:ideals}
		\xymatrix{&&0\ar[d]\\
			0\ar[r]&\c I_{T/\bP}\ar[r]&\c I_{X/\bP}\ar[r]\ar[d]&\c I_{X/T}\ar[r]& 0 \\
			&&\c I_{C/\bP}\ar[d]\\
			&&\c I_{D/\Pi}\ar[d]\\
			&&0
		}   
	\end{equation}
	Recall that the plane $\Pi$ was characterized by $\Pi = \bP W^\vee$, where $W$ is the 3--dimensional vector space defined by the exact sequence:
	\begin{equation}
		\label{eqn:W}
		0 \longrightarrow H^0(C,K_C) \longrightarrow H^0(C,L) \longrightarrow W \longrightarrow 0.
	\end{equation}
	
	Consider the following diagram with exact rows (cf. \cite{ABW}, Section V):
	\begin{equation}\label{E:long}
		\xymatrix@C=0.6cm{
			\bigwedge^2H^0(K_C)\ar[r]\ar[dr]^-0&H^0(K_C)\otimes H^0(L)\ar[r]^-u\ar[d]&\Sym^2H^0(L)\ar[r]^-v\ar[d]&\Sym^2W\ar[d]\ar[r]& 0\\
			0\ar[r]&H^0(2K_C+D)\ar[r]&H^0(2L)\ar[r]&H^0(\cO_D(2L))\ar[r]& 0
		}    
	\end{equation}
	and let  $U := \mathrm{Im} (u)=\ker(v)$.  The elements in $U$ are quadratic forms $Q\in \Sym^2V$ that can be written as $Q=\sum \ell_i\ell'_i$
	with $\ell_i,\ell'_i\in V$, such that all $\ell_i$ vanish at $D$ and hence vanish on $\Pi$.
	
	We obtain a map naturally induced by restriction to $C$
	\[
	m: U \longrightarrow H^0(2K_C+D).
	\]
	
	\begin{lemma}\label{L:kerm}
		We have
		\[
		\ker(m) = H^0(\bP,\c I_{X/\bP}(2)).
		\]
	\end{lemma}
	
	\begin{proof}
		Diagram \eqref{E:long} induces the following one:
		\[
		\xymatrix{
			0\ar[r]& U \ar[r]\ar[d]^-m&\Sym^2H^0(L)\ar[r]^-v\ar[d]^-a&\Sym^2W\ar[d]^-b\ar[r]& 0\\
			0\ar[r]&H^0(2K_C+D)\ar[r]&H^0(2L)\ar[r]&H^0(\cO_D(2L))\ar[r]& 0
		}
		\]
		Observing that $\ker(a)= H^0(\c I_{C/\bP}(2))$ and $\ker(b)=H^0(\c I_{D/\Pi}(2))$ and comparing with the vertical sequence in \eqref{E:ideals} we conclude. 
	\end{proof}
	
	Now consider the upper left section of diagram \eqref{E:long}. We have:
	$$
	\xymatrix{
		\bigwedge^2H^0(K_C)\ar[r]\ar@{=}[d]&H^0(K_C)\otimes H^0(K_C)\ar[r]\ar[d]&\Sym^2H^0(K_C)\ar[d]\ar[r]&0\\
		\bigwedge^2H^0(K_C)\ar[r]&H^0(K_C)\otimes H^0(L)\ar[r]^-u\ar[d]&U\ar[r]\ar[d]&0\\
		&H^0(K_C)\otimes W\ar[d]\ar@{=}[r]&H^0(K_C)\otimes W\ar[d]\\
		&0&0
	}
	$$
	and therefore from the last column we deduce the following:
	\begin{equation}
		\label{eqn:ker(p)}
		\xymatrix{
			0\ar[r]&\Sym^2H^0(K_C)\ar[d]^-n\ar[r]&U\ar[r]\ar[d]^-m&H^0(K_C)\otimes W\ar[r]\ar[d]^-p&0\\
			0\ar[r]&H^0(2K_C)\ar[r]&H^0(2K_C+D)\ar[r]&H^0(\cO_D(2K_C+D))\ar[r]&0
		}
	\end{equation}
	
	\begin{lemma}\label{L:kerp}
		We have
		\[
		\ker(p) = H^0(\c I_{X/T}(2)).
		\]
	\end{lemma}
	
	\begin{proof}
		Observing that $\ker(n) = H^0(\c I_{T/\bP}(2))$ and recalling Lemma \ref{L:kerm} we conclude, after comparing with the horizontal sequence in \eqref{E:ideals}.
	\end{proof}

	Consider a basis $x_0\ldots,x_{g+2}$ in $V$ such that $x_0,\ldots,x_{g-1}$ is a basis in $H^0(K_C)\subset V$. The images $w_g,w_{g+1},w_{g+2}$ of $x_g,x_{g+1},x_{g+2}$, respectively, in $W$ form a basis and the 2--plane $\Pi=\bP W^\vee$ is given by the equations $x_0=\cdots=x_{g-1}=0$. Having identified $H^0(\c I_{X/T}(2))$ with a subspace of $H^0(K_C)\otimes W$, any element $q\in H^0(\c I_{X/T}(2))$ can be written as $q=\sum_{i=g}^{g+2} \ell_ix_i|_T$, with $\ell_i\in V$ depending only on the variables $x_0,\ldots,x_{g-1}$. From the diagram (\ref{eqn:ker(p)}) and taking into account that the map $\Sym^2H^0(K_C)\to H^0(2K_C)$ is surjective, we infer that the condition that $q\in \ker(p)$ translates into the existence of a relation of the following type over $C$: 
	\[
	\sum_{i=g}^{g+2} \ell_ix_i|_C=\sum_{i=0}^{g-1}\ell'_ix_i|_C,
	\]
	with $\ell'_i$ linear forms depending only on $x_0,\ldots,x_{g-1}$.
	
	
	\medskip
	
	
	\medskip
	
	We can prove now that the elements of a given a basis in $H^0(\c I_{X/T}(2))$ are involved in elements of $\ker(\mu_{X/T})\subset H^0(\c I_{X/T}(2))\otimes V$, as claimed. 
	
	\medskip
	
	Observe that 
	\[
	H^0(K_C-D)\otimes W \subset \ker(p)=H^0(\c I_{X/T}(2))\subset H^0(K_C)\otimes W
	\]
	Therefore we have the following commutative and exact diagram:
	\[
	\xymatrix@C=0.7cm{
		0\ar[r]&H^0(K_C-D)\otimes W \ar[r]\ar@{=}[d]&H^0(\c I_{X/T}(2))\ar[r]\ar[d]&H\ar[d]\ar[r]&0\\
		0\ar[r]&H^0(K_C-D)\otimes W \ar[r]&H^0(K_C)\otimes W\ar[r]\ar[d]^-p&Z\otimes W\ar[r]\ar[d]^-{p'}&0\\
		&&H^0(\mathcal{O}_D(2K_C+D))\ar@{=}[r]&H^0(\mathcal{O}_D(2K_C+D))}
	\]
	which defines both $H$ and $Z$.  As mentioned earlier, we consider elements of $H^0(K_C-D)\otimes W$ and of $H$ separately. Choosing a splitting, we can see $H$ as a subspace of $H^0(\c I_{X/T}(2))$.
	\medskip
	
	Let $q=\ell x_i|_T \in H^0(K_C-D)\otimes W$ and let $q'=\ell' x_i|_T$ for some  $\ell,\ell'\in H^0(K_C-D)$, $\ell\ne \ell'$.  Then also $q'\in H^0(K_C-D)\otimes W$ and:
	\[
	{\mu_{X/T}}(q\otimes \ell' - q'\otimes \ell) =\ell x_i\ell'|_T -\ell'x_i\ell|_T =0.
	\]
	Therefore $q$ is involved in a non-trivial element of $\ker({\mu_{X/T}})$.
	
	\medskip
	
	Let $q=\sum_{i=g}^{g+2} \ell_i x_i|_T\in H$ for some $\ell_i\notin H^0(K_C-D)$. Choose $0\ne \ell\in H^0(K_C-D)$ and consider $q_i=\ell x_i|_T$. Then 
	\[
	{\mu_{X/T}}\left(q\otimes \ell - \sum_{i=g}^{g+2} q_i\otimes \ell_i\right) = \left(\sum_{i=g}^{g+2} \ell_ix_i|_T\right)\ell - \sum_i(\ell x_i|_T)\ell_i =0
	\]
	and hence this element $q$ is involved in a non-trivial element of $\ker({\mu_{X/T}})$, too. The proof of Theorem \ref{T:syzX} is now complete.


	
	
	
	
	%
	%
	%
	

	\section{The bielliptic case}
	\label{sec:excepted}
	
	In this section, we give a partial answer in the bielliptic case that was excluded previously. We assume $g \ge 11$, and $C$ is bielliptic, $d=2g+2$ and $M$ is a tetragonal pencil on $C$. Note that $M$ is composed with the bielliptic involution.
	We prove the following.
	
	\begin{thm}
		Notation as above. If $C$ is embedded via $L=K_C+D$ with $D\in C_4$ general, then $\mathrm{Syz}_2(C)=C$.   
	\end{thm}
	
	\begin{proof}
		
		In view of Theorem \ref{thm:semicontinuity} it is enough to treat the special case in which $D \in |M|$.
		Let $S \subset \bP^{g+2}$ be the scroll associated to $M$; let $\cO_S(1)=\cO_S(H)$ and $F$ a fibre of the projection $\pi:S \to \bP^1$.
		Let $\c I_{C/S}$ be the ideal sheaf of $C$ in $S$. We have 
		$$ \pi_*(\c I_{C/S}(2H))=\cO_{\bP^1}(a_1)\oplus \cO_{\bP^1}(a_2)$$
		with $a_1 \le a_2 $ and $a_1+a_2=g-1$.
		
		The line bundle $\cO_S(H-F)$ restricts on $C$ to $K$ and it maps $S$ as the scroll associated to $|M|$ in the canonical model.
		From Theorem 10 in \cite{ABS19}, we have that 
		$$ \pi_*(\c I_{C/S}(2H-2F))=\cO_{\bP^1}\oplus \cO_{\bP^1}(g-5)$$
		so that $a_1=2$ and we conclude with the rolling factors trick, as in Lemma 5 of~\cite{ABS19}.
	\end{proof}

	Note that the proof went by specialization to the most degenerate case, when $D\in |M|$, and however, the statement holds generically.

\end{document}